\documentclass[12pt,a4paper]{amsart}

\usepackage{mathrsfs,amsmath, amscd, amsthm,amssymb, amsfonts, verbatim,subfigure, enumerate,slashed}
\usepackage{graphicx}
\usepackage[top=1in, bottom=1.25in, left=1in, right=1in,footskip=.5in]{geometry}

\usepackage{amsaddr}
\usepackage{enumitem}

\usepackage[colorlinks=false]{hyperref}

\hypersetup{hidelinks}

\usepackage{fouriernc}
\usepackage{cite}
\usepackage{tikz-cd}

\newtheorem{theorem}{Theorem}[subsection]
\newtheorem{definition}[theorem]{Definition}

\newtheorem{proposition}[theorem]{Proposition}

\newtheorem{remark}[theorem] {Remark}     
\newtheorem{example}[theorem]{Example}

\newcommand{\Lm}{\mathbb L}
\newcommand{\var}{\it Var}
\newcommand{\rk}{\rm rk}
\newcommand{\sgn}{\rm sgn}

\begin{document}

\author{Felipe Espreafico$^{\flat,\natural}$ and Johannes Walcher$^{\natural}$}

\title{On Motivic and Arithmetic Refinements of Donaldson-Thomas invariants}

\email{felipe.espreafico@impa.br}
\email{walcher@uni-heidelberg.de}

\date{May 20, 2023}

\thanks{
J.W.’s partial funding statement reads: This work is funded by the Deutsche Forschungsgemeinschaft (DFG, German 
Research Foundation) under Germany’s Excellence Strategy EXC 2181/1 — 390900948 (the Heidelberg STRUCTURES 
Excellence Cluster).F.E.'s funding statement reads: This work is funded by CAPES, under process number 88881.690115/2022-01}

\address{
{$^\flat$}Institute for Pure and Applied Mathematics \\
Rio de Janeiro  - RJ, 22460-320, Brazil \\[.2cm]
{$^\natural$}Institute for Mathematics, Heidelberg University \\ 69120 Heidelberg, Germany}

\begin{abstract}
In recent years, a version of enumerative geometry over arbitrary fields has been 
developed and studied by Kass-Wickelgren, Levine, and others, in which the counts obtained are not integers but quadratic forms. Aiming to understand the relation to other "refined invariants", and especially their possible interpretation in 
quantum theory, we explain how to obtain a quadratic version of Donaldson-Thomas invariants
from the motivic invariants defined in the work of Kontsevich and Soibelman and pose some questions. We calculate these invariants in a few simple examples that provide standard tests for these questions, including degree zero invariants of $\mathbb A^3$ and higher-genus Gopakumar-Vafa invariants recently studied by Liu and Ruan. The comparison with known real and complex counts plays a central role throughout.
\end{abstract}

\maketitle
 
\section{Introduction}

Over the years, it has been a fruitful endeavour to elucidate the structure of enumerative invariants
attached to various classes of problems stemming from algebraic geometry, symplectic geometry, and
theoretical physics, see e.g.\ \cite{thirteenhalves}. One of the common themes in this context has been the 
identification of problems attached to distinct geometric or algebraic setups ("duality"). 
Such an identification can uncover otherwise unexpected properties, such as regularity, 
recursions, or modularity, and often allows for exact solutions. More recently, it has been of interest to 
compare invariants attached to the same geometry, but defined in different theories, relying on higher
structures or hidden symmetries ("refinement"). It is ultimately reasonable to expect, if not outright necessary 
for coherence, that any refinement or structure present in one setup or theory can be ported to all the others, 
and any missing piece is as much an opportunity for further progress.

In this note, we will be concerned mainly with the youngest type of "arithmetic invariants"  which
are defined in $\mathbb A^1$-enumerative geometry and take values in the Grothen\-dieck-Witt ring of
quadratic forms, and the somewhat older "motivic invariants", which take values in the Grothendieck ring
of varieties. We briefly recall their salient features here, and their more proper definition in the main
text.

$\mathbb A^1$-Enumerative geometry consists in applying the machinery of so-called 
$\mathbb A^1$-homotopy theory to classical enumerative problems. This machinery was introduced in the 90s, in the 
work of Morel and Voevodsky \cite{morel_1-homotopy_1999}. The main idea is to consider the affine line $\mathbb A^1$ 
rather than the interval $[0,1]\subset \mathbb R$ as a parameter space and thereby construct a homotopy theory for schemes. This allows 
us to define an algebraic or $\mathbb A^1$-degree for maps between schemes, which naturally generalizes the classical 
degree from algebraic topology. While the natural ring where this degree is defined is not $\mathbb Z$ but rather 
the Grothendieck-Witt ring of quadratic forms, the theory has the highly attractive feature of being defined over arbitrary
fields, rather than just the complex or real numbers. The latter, numerical, counts can be recovered as special cases from rank and 
signature of the quadratic form. Problems which have been solved in this theory include the classical enumeration of lines
on a cubic surface \cite{kass_arithmetic_2021}, lines on a quintic threefold \cite{pauliquintic} and the enumerative geometry of
twisted cubics \cite{levine_quadratic_2022}. The latter study in particular relies on Atiyah-Bott localization, which was 
extended to this context in \cite{levine_atiyah-bott_2022}. The results for the quintic specialize to the real
counts first obtained in \cite{solomonintersection} and \cite{walcher2007}. Even more recently, a refined enumeration of rational curves through points in a del Pezzo surface was obtained in \cite{Kass_Levine_Solomon_Wickelgren_2023}, which in particular refines the classical counts of degree $d$ rational curves in $\mathbb P^2$ trough $3d-1$ points. 

Motivic Donaldson-Thomas (DT) invariants were introduced from a rather different direction by Kontsevich and Soibelman
\cite{kontsevich_stability_2008}. Their main motivation was to understand wall-crossing phenomena in the theory of
BPS states and specifically the enumerative geometry of Calabi-Yau 3-folds. The most basic observation in this context
is that enumerative invariants arise from a perfect obstruction theory on a moduli space (equipped with a stability 
condition), and the basic idea is to identify that algebraic obstruction class directly as a "motivic invariant", while 
the standard "numerical" DT-invariants can be recovered via the Euler Characteristic morphism (or motivic integration).

Apart from circumstantial evidence in the B-model \cite{krefl2010}, an early indication that the "refined" theory 
might be related to an enumerative geometry over fields other than the complex numbers appeared in work of
G\"ottsche and Shende \cite{Goettsche_Shende_2014} on the enumerative geometry of curves on surfaces. The main
observation in this regard was that the real count, defined via the compactly supported Euler characteristic, 
appears in a special limit of the refinement parameter that is different from the standard complex one.

Our goal here is to shed light on the relationship between these two generalizations of classical enumerative 
invariants, and on how they arise by composing the complex and real counts, which are either known or accessible 
by more standard methods. The basic mathematical idea is to utilize the compactly supported $\mathbb A^1$-Euler 
characteristic 
$$
\chi_c^{\mathbb A^1}:K_0(\var(k))\to GW(k) \,,
$$
which is already considered for a similar purpose in the papers \cite{arcila-maya_compactly_2022} and 
\cite{Röndigs_2016}, to go from motivic to arithmetic invariants, in a way that is compatible with taking
numerical degrees on the two sides. We illustrate the discussion by two examples which to our knowledge have
not been studied in the arithmetic context before. The first is the Hilbert scheme of points on $\mathbb A^3$, for 
which the motivic DT invariants were computed in the paper \cite{behrend_motivic_2013}. Using their formula for 
the motivic partition function of the Hilbert scheme of points of $\mathbb C^3$ and passing to the arithmetic 
invariants, we get the expected result over $\mathbb R$ (as computed in \cite{krefl_real_2010}) and, as a byproduct,
a count that would make sense over any field. Our second example is based on the recent calculation of Gopakumar-Vafa
invariants at the Castelnuovo-Mumford bound by Liu-Ruan \cite{liu_castelnuovo_2022} and S.\ Feyzbakhsh in \cite{alexandrov2023quantum}
and also leads to a prediction  for a count over any field.

In all these examples, the knowledge of the real (and of course, the complex) count is sufficient to determine the 
arithmetic invariant (i.e., the value of the invariant over any field). On the other hand, the real enumerative 
invariants admit a physics interpretation, which mimics the relationship between real and open Gromov-Witten theory, 
as "open" or "relative" BPS invariants. This is indicative of a much deeper relationship between arithmetic refinements
of enumerative geometry and the physics of BPS states. In fact, our interest in the arithmetic invariants arises 
in part from the observations, recorded in \cite{walcher2012arithmetic} and otherwise, that enumerative predictions of
mirror symmetry depend in general on an arithmetic structure of the mirror manifold, and that this is especially
true in the context of relative invariants. These observations have not been explained from the motivic point of
view. 

An initial idea for the physics interpretation of the quadratic form is the occurence of the Eisenbud–Levine–Khimshiashvili 
class as local $\mathbb A^1$ degree \cite{wickelgrenEisenbud}. In such examples, the function whose critical
locus defines the enumerative invariant enters as superpotential into the action functional of the physical theory.
A critical point defines a supersymmetric vacuum, and a refinement via the Hessian would correspond to keeping track of 
the masses of the physical fields.

To summarize, the main purpose of this paper is not so much to claim any new fundamental insight, but merely to
point out that the arithmetic and motivic refinements are too closely related for this to be an accident, that this applies
specifically in a class of examples that are highly relevant in the context of BPS state counting and M-theory dualities,
and thereby to pose the question for a physics interpretation of all these structures.

\subsection*{Acknowledgments} F.E. would like to thank to Sabrina Pauli and Stephen McKean for enlighting discussions about $\mathbb A^1$-enumerative geometry.

\section{Motivic DT invariants}
\subsection{The Grothendieck ring of varieties}
Let $k$ be a field of characteristic zero. 

\begin{definition}
The Grothendieck ring of varieties is the abelian group generated by isomorphism classes of all varieties over $k$ 
modulo relations 
$$[X] = [Y] + [X\setminus Y],$$ 
where $Y$ is closed in $X$. 
		
The product structure that makes $K_0(\var(k))$ a ring is given by the usual Cartesian product $[X]\cdot [Y] = [X\times Y]$. 
The class of the affine line $[\mathbb A_k^1]$ is denoted by $\mathbb L$ in $K_0(\var(k))$ and we set 
$$\mathcal M_k:=K_0(\var(k))[\mathbb L^{-\frac12}],$$
which will be called ring of motivic weights.
\end{definition}
	
	 \begin{remark}
	 	The following facts are true:
	 \begin{itemize}
	 	\item If $f:X\to S$ is a Zariski locally trivial fibration with fiber $F$, then $[X] = [S]\cdot [F]$.
	 	\item If $f:X\to Y$ is a bijective morphism, it is true that $[X] = [Y]$.
	 \end{itemize} 
	 \end{remark}

With this notation, the classes of cellular varieties can be computed in terms of 
$\mathbb L$, for example,
$$
[\mathbb P^n] = 1 + \mathbb L + \mathbb L^2 + \dots + \mathbb L^n
$$
and a similar expression can be written for the Grassmanians:
$$
[Gr(n,k)] = \frac{(\Lm^n-1)(\Lm^{n-1}-1)\cdots(\Lm-1)}{(\Lm^{n-k})(\Lm^{n-k-1})\cdots(\Lm-1)\cdot(\Lm^k-1)(\Lm^{k-1}-1)\cdots(\Lm-1)}.
$$

  \begin{remark}
      Notice that, over non algebraically closed fields, 
      besides the rational points, varieties have also points which correspond to Galois orbits over the algebraic closure. 
      This gives rise to different classes of points in $K_0(\var(k))$, which are the classes of ${\rm Spec(L)}$, with $L\supset k$.
  \end{remark}

  \begin{example}
      Over $\mathbb R$, we have two classes of points: $[{\rm Spec}\mathbb R]$ and $[{\rm Spec}\mathbb C]$. Therefore, 
      although $\mathbb P^1(\mathbb R)$ is a circle, there are non rational points which correspond to pairs of 
      complex conjugate points. Indeed, $\mathbb P^1$ is not isomorphic to $X = \{x^2 + y^2 =1\}\subset \mathbb A^2$, 
      since there are two complex conjugate points "at infinity". In terms of the Grothendieck ring of varieties, we 
      get $[X]=\Lm+1 - [{\rm Spec}\mathbb C]\in K_0(\var(\mathbb R))$.
      \end{example}
Over $k=\mathbb C$, there is a morphism to $\mathbb Z$ given by the Euler characteristic with compact support. 
$$
\chi_c:K_0(\var(\mathbb C))\to \mathbb Z
$$
which can be extended to a morphism\label{extend}
$$
\chi^{\mathbb C}_c:\mathcal M_{\mathbb C}\to \mathbb Z
$$
if we send $\mathbb L^{\frac 12}$ to $-1$ (since $\mathbb L$ goes to 1). The same can be done for considering the 
Euler characteristic of the real points. In Section \ref{A1}, we will introduce the $\mathbb A^1$-Euler characteristic, 
which is defined for varieties over any field. Notice that we have to be careful when extending to $\mathcal M_{\mathbb R}$. 
As $\chi^{\mathbb R}_c(\mathbb L) = -1$, the extended Euler characteristic takes values in $\mathbb Z[i]$, where $i=\sqrt{-1}$. 
$$
\chi^{\mathbb R}_c:\mathcal M_{\mathbb R}\to \mathbb Z[i]
$$

If one does not want to consider compactly supported Euler characteristics, one could actually only consider projective 
varieties. In fact, the abelian group generated by projective varieties with relations given by
\begin{align}
    \label{blowup}
    [X] - [Y] &= [{\rm Bl}_YX] - [E]\\
    [\varnothing] &= 0
\end{align}
where ${\rm Bl}_YX$ is the blow up of $X$ along $Y$ and $E$ is the exceptional divisor, is isomorphic to $K_0(\var(k))$. 
For details, see Theorem 3.1  of \cite{bittner_universal_2004}.


 \subsection{Donaldson-Thomas invariants} \label{dtinv}

We consider Donaldson-Thomas invariants following \cite{Maulik_2006}. Those are simply numbers\footnote{For emphasis, one sometimes
refers to them as "numerical DT-invariants", to distinguish them from their "motivic" version, defined below.} obtained by integration against 
virtual classes of moduli spaces of ideal sheaves in the complex setting. If $M$ is a smooth Calabi-Yau threefold, any ideal 
sheaf defines a subscheme $N$. Then, the moduli spaces we are interested in are the ones of the form $X=I_n(M,\beta)$ and 
consist of sheaves for which $N$ represents the homology class $\beta$ and $\chi(\mathcal O_N) = n$. Usually, we are 
interested in the cases for which $N$ has dimension at most one, which implies that $\beta$ is in $H_2(M)$. $I_n(M,\beta)$ is, 
then, isomorphic to a Hilbert Scheme of curves in $M$. In particular, for the case $\beta = 0$, we are considering exactly 
the situation in which $N$ has dimension 0, that is, $I_n(M,0)$ is the Hilbert Scheme of points of $M$.

The DT invariants of $M$ are defined by integrals:
\begin{equation}
    DT(n,\beta) = \int_{[I_n(M,\beta)]^{{\it vir}}}1
\end{equation}
where $[I_n(M,\beta)]^{vir}$ is a virtual class in the homology defined via obstruction theory. It was proved, in 
\cite{Behrend_2009}, to depend only on the scheme structure of the moduli space and not on the chosen obstruction theory. 
In particular, if the moduli space ends up being smooth, this virtual class is, up to a factor $(-1)^{\dim I_n(M,\beta)}$, 
simply the Poincaré dual of the top Chern class, which means that the integral above is given, up to sign, by the Euler characteristic 
with compact support of the moduli space. 

A natural way of refining such invariants is, therefore, to consider these virtual classes as elements in the Grothendieck 
ring of varieties, since the Euler characteristic gives us a natural morphism from this ring to $\mathbb Z$. If the moduli 
space is smooth these virtual classes should be given by the classes of $I_n(M,\beta)$ in the ring $\mathcal M_{\mathbb C}$, 
since the DT invariants coincide with the Euler characteristic. Notice, though, that, if the moduli space is singular, it 
is usually not straightforward to define such virtual classes in the Grothendieck ring of varieties.  Pictorially, we want 
a commutative diagram as below. 
\begin{equation}
\begin{tikzcd}
\label{dt-diag}
\left\{{\rm Moduli\ Spaces\ }I_n(M,\beta)\right\}\arrow[dddd, dashed,"\text{Virtual Classes}" left]\arrow[rdd,"\text{DT Theory}"]\\
\\
&\mathbb Z \\
\\
\mathcal M_{\mathbb C}\arrow[uur,"\chi_c" below]
\end{tikzcd} 
\end{equation}

The idea of refining DT invariants via the Grothendieck ring of varieties was first considered in \cite{kontsevich_stability_2008}. 
For us, the interest behind this is that, as already pointed out in the introduction and in \cite{Goettsche_Shende_2014}, by 
considering real and complex Euler characteristics, we can, in some cases, recover real DT invariants from the same formulas 
in the Grothendieck ring. In general, the virtual classes are not in $\mathcal M_k$ (or $\mathcal M_{\mathbb C}$), but in a 
slightly bigger equivariant version of this ring. In the next section, we present a way to define virtual classes over a 
general field $k$ of characteristic zero when the moduli spaces are defined over $k$ and can be represented as a critical 
locus of a smooth function, based mainly in the work of Denef and Loser \cite{Denef_Loeser_2002}. In practice, this is 
the case in many interesting examples in which $D^b(X)\cong D^b(\mathbf{Q}-{\rm mod})$, that is, in which we can represent 
the moduli spaces as quiver varieties.

\subsection{The motivic nearby class}\label{motivic-nearby-cycles}  Let $X$ be a smooth variety and $f:X\to \mathbb A^1$ be 
a regular function. Denote the central fibre $f^{-1}(0)$ by $X_0$. Our goal is to define a class $[Z]_{{\it vir}}$ for $Z = \{df = 0\}$, 
the critical locus of $f$, which encodes also the scheme structure of $Z$. 
\begin{definition}
Let $\mu_n$ be the group of $n-$roots of unity in $\bar k$. Notice that it has a structure of algebraic variety over $k$. 
We have maps $\mu_{nd}\to \mu_n$ given by $x\mapsto x^d$. This gives us a projective system. We denote the limit by $\hat\mu$. 
A good $\mu_n$ action on $X$ is a group action $\mu_n\times X\to X$ which is a morphism of varieties and such that each orbit 
is contained in an affine subvariety of $X$. A good $\hat\mu$-action is a group action $\hat\mu\times X\to X$ which factors 
through a good action of $\mu_n$ for some $n$.
\end{definition}

\begin{definition}
Consider the abelian group generated by symbols $[X, \hat \mu] = [X]$, where $X$ is a varitey with a good $\hat\mu$-action, 
modulo isomorphisms compatible with the action. The relations are the same scissor relations $[X] = [Y] + [X\setminus Y]$ 
that we have for $K_0(\var(k))$ but such that the action on $Y$ is induced by the one in $X$. Finally, there is one more 
relation given by $[X\times V] = [X\times \mathbb A^n]$, where $V$ is the affine space of dimension $n$ with a good $\hat\mu$-action 
and, on the other side, $\mathbb A^n$ has the trivial action. The product is given by the usual Cartesian product with the product 
action. The class of the affine line with the trivial action is denoted by $\mathbb L$. This ring will be denoted 
$K_0(\var^{\hat\mu}(k))$ and its localization $K_0(\var^{\hat\mu}(k))[\mathbb L^{-\frac12}]$ by $\mathcal M^{\hat\mu}_k$.
\end{definition}

Take a resolution of $f:X\to \mathbb A^1$, that is, a map $h:Y\to X$, with $Y$ smooth and irreducible such that $Y_0 = h^{-1}(X_0)$ 
has only normal crossings and the restriction $h: Y-Y_0\to X-X_0$ is an isomorphism. We fix the notation:
\begin{itemize}
\item $E_i$, $i\in J$ denote the components of $Y_0$ and $N_i$ their multiplicities;
\item $E_I$, $I\subset J$ denote the intersections $\bigcap_{i\in I}E_i$;
\item $E_I^{\circ}$, $I\subset J$ denote $E_I - \bigcup_{j\notin I} E_j\cap E_I$
\end{itemize}
To take multiplicities into account, we define natural covering spaces $\tilde E_I^{\circ}\to E_I^{\circ}$, which are 
unramified and Galois with Galois groups given by $\mu_{m_I}$, where $m_I$ is the greatest common divisor of the multiplicities 
of all $E_i, i\in I$. 
Take $U\subset Y$ an open set such that $f\circ h = uv^{m_I}$, where $u$ is a unit in $\mathcal O_Y(U)$ and $v$ is a 
morphism $v:U\to k$. We can define $\tilde E_I^{\circ}$ by gluing the sets
$$
\{ (x,t)\in (U\cap E^{\circ}_I)\times\mathbb{A}^1 | t^{m_I} = u^{-1}(x) \}.
$$
The considerations above show that $\tilde{E}^{\circ}_I$ is a Galois covering and that there exists a natural good action 
of $\mu_n$ (and thus, of $\hat\mu$) on it.

\begin{definition}
	The motivic nearby class of $f$ is defined as follows:
	\begin{equation}
		\label{motivicnearby}
	S_f:=\sum_{\varnothing\neq I}(1-\mathbb L)^{|I| - 1}[\tilde E_I^o]\in \mathcal M_k^{\hat\mu}
	\end{equation}
\end{definition}
Among other remarks, the belief is stated in \cite{Denef_Loeser_2002} that this is a motivic incarnation of the 
complex of nearby cycles of $X_0$, which was defined in Exposé XIII of \cite{SGA7}. Restricting $S_f$ to a point 
$x\in X_0$, we get the local version, which the authors believe to be a motivic version of the classical Milnor fibre of $f$ over $x$.
\begin{equation}
		\label{localmotivicnearby}
	S_{f,x}:=\sum_{\varnothing\neq I}(1-\mathbb L)^{|I| - 1}[\tilde F_I^o]\in \mathcal M_k^{\hat\mu}
	\end{equation}
where $F_i$ is the fiber of $E_i$ over $x$.

\begin{example}\label{mfex}
Let $f:\mathbb A^2\to k$ be given by $f(x,y) = x^2-y^2$. Then $X_0$ has already only normal crossings, which are the two lines $E_1$ and $E_2$ intersecting at the origin. We have:
\begin{align*}
&E_1^{\circ} = E_1 - 0\\
&E_2^{\circ} = E_2 - 0\\
&E_{1,2}^{\circ} = 0
\end{align*}

As all multiplicities are equal to 1, we do not need to worry about the covering spaces. Computing $S_f$, we get:

\begin{equation}
\label{mfgex}
  S_f = [E_1 - 0] + [E_2 - 0] + (1-\Lm)[0] = \Lm - 1 + \Lm - 1 + 1-\Lm = \Lm-1  
\end{equation}

For $S_{f,0}$, we can write:
\begin{equation}\label{mflex}
 S_{f,x} = [(E_1 - 0)\cap h^{-1}(0)] + [(E_2 - 0)\cap h^{-1}(0)] + (1-\Lm)[0\cap h^{-1}(0)] = 1-\Lm   
\end{equation}

\end{example}

It is instructive to compare these results with what happens over $\mathbb R$ and $\mathbb C$. The Milnor fibre over a 
point $x\in X_0$, over $\mathbb C$, is defined as the intersection of the preimage $f^{-1}(w)$, for small $w\in\mathbb C$, 
with a small ball around $x$. This gives us a cylinder around $x$, which has Euler characteristic 0: this is what we get 
by making $\Lm = 1$ in equation \ref{mflex}. Considering the same definition over $\mathbb R$ (which cannot be done in 
all cases, considering that the topology of the Fiber $f^{-1}(w)$ can change depending on $w$), we get two hyperbola 
segments. As we are intersecting with a ball, those are compact and the Euler characteristic should be 2, which is 
compatible with making $\Lm = -1$ in equation \ref{mflex}.  If we do not intersect with the ball (that is, consider, 
in some sense, a global fibre), we would get an infinite cylinder ($\chi_c = 0$) or two non compact branches of 
hyperbola($\chi_c = -2$). This is compatible with equation \ref{mfgex}.

As we will point out in Section \ref{A1}, $\mathbb A^1$-enumerative geometry has an arithmetic incarnation of the Milnor number for isolated singularities, which is the EKL class of the gradient of $f$. Although, in order to compute DT invariants, we use the definition above in the case of non isolated singularities, an interesting problem would be to understand the relationship between Milnor numbers and the class $S_{f}$. We elaborate on this in Section \ref{A1}.

\subsection{Virtual class of a critical locus} \label{virtclasses}
\begin{definition}\label{virtclassdef}
   Let $X$ be a smooth variety and $f:X\to \mathbb A^1$ be a regular function. If $Z$ is the critical locus of $f$ and $X_0 = f^{-1}(0)$, the virtual class of the critical locus $Z$ is defined as
   \begin{equation}
        [Z]_{{\it vir}} = -\mathbb L^{-\frac{\dim X}{2}}(S_f-[X_0])\in\mathcal M_k^{\hat\mu}
   \end{equation}
\end{definition}

The idea behind this definition is that, over the smooth points of $X_0$, the Milnor fibre should correspond to $X_0$ itself. This means that this virtual class encodes exactly what is happening on the singular locus of $X_0$. Notice that this definition depends on the function $f$ chosen, that is, it depends on how the variety is presented as a critical locus.

\begin{example}\label{smoothvirtual}
    If $f = 0$ is the zero function, then $Z = X_0 = X$ (every point is critical). We can compute $S_f = 0$ using the fact that $X_0$ can be seen as the zero divisor. Therefore:
    $$
    [Z]_{{\it vir}} = -\Lm^{-\frac{\dim X}2}(0 - [X_0]) = \Lm^{-\frac{\dim X}2}[X]
    $$
\end{example}
This example shows that the virtual class of a smooth variety is its class in $\mathcal M_k$ up to multiplication by a factor. In terms of $DT$ invariants, this factor correspond to the sign difference that we get with respect to the Euler characteristics.

The virtual class is not always in $\mathcal M_k$ but, over $\mathbb C$, Behrend, Bryan and Szendröi showed, in \cite{behrend_motivic_2013}, that, when $f$ is equivariant with respect to a torus action on $X$, $[Z]_{{\it vir}}\in \mathcal M_k$ and can be computed from the difference between the zero fiber and the generic fiber.

\begin{proposition}[cf. Thm. B.1 in \cite{behrend_motivic_2013}]
\label{torus}
Let $f:X\to \mathbb C$ be a regular morphism on a smooth quasi-projective variety. Let $Z$ be 
the critical locus of $f$. Assume that there exists an action of a connected complex torus on $X$ 
such that $f$ is equivariant with respect to a primitive character. If there is a one parameter 
subgroup $\mathbb C^{\times}\subset T$ such that the induced action is circle compact, that 
is, the set of fixed points is compact and the limit $\lim_{\lambda\to 0}\lambda x$ exists 
for all $x\in X$, then, the virtual class $[Z]_{{\it vir}}$ is given by
$$
[Z]_{{\it vir}} = -\mathbb L^{-\frac{\dim X}{2}}([X_1]-[X_0])\in \mathcal M_{\mathbb C},
$$
where $[X_1]$ is the class of the fibre $f^{-1}(1)$. 
\end{proposition}

The result above strongly uses the fact that the varieties are defined over $\mathbb C$, but our computations in 
section 4 show that, at least in the case of degree zero invariants of $\mathbb A^3$, the real DT invariants are 
also encoded in the formula. Inspired by the work of Levine on localization formulas in $\mathbb A^1$-enumerative 
geometry \cite{levine_atiyah-bott_2022}, one may think that it could be possible to extend these arguments to 
more general fields.


\section{\texorpdfstring{$\mathbb A^1$}{A1}-Enumerative Geometry}\label{A1}  

\subsection{Rings of Quadratic forms} 
In this section, $k$ continues to be a field of characteristic zero.
\begin{definition} The Grothendieck-Witt ring is the group completion of the set of all quadratic forms over $k$ up to isometries 
with the operations 
		\begin{equation}
		q+q':V\oplus W\to k\qquad q+q'(x,y)=q(x)+q'(y)
		\end{equation}
		\begin{equation}
		    qq':V\otimes W\to k\qquad qq'(x\otimes y)=q(x)q'(y)
		\end{equation}
where $q:V\to k$ and $q':W\to k$ are quadratic forms representing isometry classes.
 
  \end{definition}
  
There are two important natural maps from $GW(k)$ to $\mathbb Z$ which are the rank and the signature.
  \begin{itemize}
      \item $\rk:GW(k)\to \mathbb Z$ computes the dimension of the vector space in which the quadratic form is defined.

      \item $\sgn:GW(k)\to \mathbb Z$ computes the signature of the quadratic form.
  \end{itemize}
 
 \begin{remark}
      If we have a field extension $k\subset L$, we can define a map $GW(L)\to GW(k)$ by simply considering composition with the trace map $Tr_{L/k}:L\to k$.
 \end{remark}

 \begin{remark}
     Since every quadratic form can be diagonalized, $GW(k)$ is generated by elements of rank 1. These are represented by forms $q(x)=ax^2$, with $a\in k^{\times}$. They are denoted $\langle a\rangle\in GW(k)$. If $a$ is a square, $\langle a\rangle =\langle 1\rangle$ in $GW(k)$, and therefore the generators of $GW(k)$ are the the elements in $k^{\times}/(k^{\times})^2$. Using this notation, we write $\mathbb H = \langle1\rangle + \langle-1\rangle$. This is called the hyperbolic form. It has the property that $\langle a\rangle \mathbb H = \mathbb H$ for any $a\in k$.
 \end{remark} 
The Grothendieck-Witt ring is suitable to refine the classical $\mathbb Z$-valued enumerative invariants (counts) defined over algebraically closed fields. Very roughly, the intuition is that we count points considering local "orientations" (as it is done over $\mathbb R$), which correspond to quadratic forms. Counting with such orientations allows us to get invariant counts over any field $k$. The ranks of these quadratic forms correspond to the size of the Galois orbit of the point and, therefore, they recover the counts over $\bar k$. The signature recovers the signed counts over $\mathbb R$ (when $k\subset \mathbb R)$, since the two square classes over $\mathbb R$ will correspond to the two possible local orientations (signs). The definition of the $\mathbb A^1$-degree of a map from $\mathbb A^n\to \mathbb A^n$ is a good example. 

\begin{definition}\label{deg}
    Let $P:\mathbb A^n\to \mathbb A^n$ be an étale morphism. The local degree of $P$ at a closed point $x$ with rational image $y=P(x)$ and which is isolated in its fiber is given by:
    \begin{equation}\label{localdeg}
        \mathrm{deg}^{\mathbb A^1}_{x}(P) = \mathrm{Tr}_{k(x)/k}\langle {\rm det} J(x)\rangle
    \end{equation}
    where $J(x)$ denotes the derivative (Jacobian) of $P$ at the point $x$ and $k(x)$ denotes the residue field of $x$. The degree of $P$ is simply the sum over all the preimages.
    \begin{equation}\label{globaldeg}
        \mathrm{deg}^{\mathbb A^1}(P) = \sum_{x\in P^{-1}(y)}\mathrm{Tr}_{k(x)/k}\langle {\rm det} J(x)\rangle
    \end{equation}
\end{definition} 

\begin{remark}
    When $f$ is not étale, the degree can still be defined using more involved machinery from homotopy theory. We refer the reader to section 2 of \cite{wickelgrenEisenbud} for a more complete exposition and to \cite{morel_1-homotopy_1999} for technical details.  
\end{remark}

Equation \eqref{localdeg} corresponds to our local "orientation" and, in \eqref{globaldeg}, we "count" all the points. The simplest example is the function $P(x) = x^2$ from $\mathbb A^1$ to $\mathbb A^1$.

\begin{example} Let $k=\mathbb Q$ and $P(x)=x^2$ in $\mathbb A^1$. The derivative is given by $J(x) = 2x$. Therefore, using equation \eqref{globaldeg} we can write, choosing $y=1$ and $y=-1$, the following formulas.
$$
        \mathrm{deg}^{\mathbb A^1}(P) = \mathrm{Tr}_{\mathbb Q(1)/\mathbb Q}\langle 2\rangle + \mathrm{Tr}_{\mathbb Q(-1)/\mathbb Q}\langle -2\rangle = \langle2\rangle+\langle-2\rangle=\mathbb H
$$
$$
 \mathrm{deg}^{\mathbb A^1}(P) = \mathrm{Tr}_{\mathbb Q(\sqrt{-1})/\mathbb Q}\langle 2\sqrt{-1}\rangle = \mathrm{Tr}_{\mathbb Q(\sqrt{-1})/\mathbb Q}\langle 1\rangle = \mathbb H
$$
In the second formula we used that $2\sqrt{-1} = (1+\sqrt{-1})^2$ and that the trace is simply the quadratic form $a+b\sqrt{-1}\mapsto a^2-b^2$.

Notice that $\rk(\mathbb H) = 2$, which is the number of preimages in the algebraic closure and that $\sgn(\mathbb H) = 0$, which is the real topological degree.
\end{example}
In the next section, we will introduce Chow-Witt groups, which are a generalization of the classical Chow groups to this context. The idea is to do intersection theory taking such "orientations" into account.

\subsection{Chow-Witt groups and Euler Characteristics} 
  
We start by recalling that, over $\bar k$, one can define sheaves $\mathcal K^M_n$, known as Milnor K-sheaves, for which
		\begin{equation}
			\label{Bloch-Kato}
			H^n(X,\mathcal K^M_n) \cong CH^n(X).
		\end{equation}
	
The idea is to replace this Chow groups by the Chow-Witt groups, defined in a way to have that the group associated to ${\rm Spec}(k)$ equals $GW(k)$. 

The Chow-Witt groups are denoted $\widetilde{CH}(X,L)$ (for $L$ a line bundle), and were first defined in \cite{barge_groupe_2000}. The easiest way to define them is to generalize Equation \ref{Bloch-Kato}, i.e, write
		$$
		H^n(X,\mathcal K^{MW}_n(L)) \cong \widetilde{CH}^n(X,L),
		$$
		for a suitable family of sheaves $\mathcal K^{MW}_n$ twisted by lines bundles. 
  
  The line bundles can be considered up to squares, that is, $\widetilde{CH}^n(X,L')\cong \widetilde{CH}^n(X,L'\otimes L^{\otimes2})$ for any $L, L'$. This is related to the fact that $GW(k)$ is generated by classes in $k^{\times}/(k^\times)^2$. We have pullbacks and pushforwards for Chow-Witt groups, but we have to be careful about the line bundles. This new technical complication is a result of considering "orientations".
		\begin{itemize}	
			\item For $f:X\to Y$, relative dimension $d$:
			\begin{itemize}
				\item Pullback:$	f^*:\widetilde{CH}^n(Y,L)\to\widetilde{CH}^n(X,f^*L)$
				
				\item Pushfoward (for $f$ proper): $f_*:\widetilde{CH}^n(X,\omega_X\otimes f^*\omega_Y\otimes f^* L)\to\widetilde{CH}^{n-d}(Y,L)$
			\end{itemize}
		\end{itemize}	
		
		These definitions give us:
		\begin{itemize}
			\item  a map $\int_X:\widetilde{CH}^n(X,\omega_X)\to GW(k)$, given by pushforwarding to the point ${\rm Spec}( k)$;
			
			\item a definition of Euler class of a vector bundle $\varepsilon(V)\in \widetilde{CH}^{r}(X, {\rm det}^{-1}(V))$, via pushfowarding the class $1_X$ by a section and then pulling it back. 
		\end{itemize}
		
		Notice that we can only compute the integral of an Euler class (that is, count the zeros of a generic section), if the determinant of the bundle is, up to squares, given by $\omega_X$. This gives a notion of (relatively) orientable vector bundles. For the tangent bundle $TX$, we have that ${\rm det}^{-1}(TX) = {\rm det}^{-1}((\Omega^1_X)^{-1}) = \omega_X$. This implies that we can compute $\int_X\varepsilon(TX)$, for any $X$.
		
		\begin{definition}
			\label{euler}
			The Euler characteristic of a smooth and projective variety $X$ can be defined as
			$$
			\chi^{\mathbb A^1}(X) := \int_X \varepsilon(TX)
			$$
		\end{definition}
		
		If $X$ is not projective, one can still define the Euler characteristic. We could consider the motivic stable homotopic category and use the fact that the infinite suspension spectrum of $X$ is strongly dualizable and gives rise to and endomorphism of the sphere spectrum, which corresponds to an element of $GW(k)$. The relationship between these two definitions is explained in \cite{levine_aspects_2020} and was first proved in  \cite{Levine_Raksit_2020}. The same idea can extend this map to a compactly supported Euler characteristic. 
		
	   With this definition in hand, we can consider the map defined on $K_0(\var(k))$ that takes any variety and evaluates its compactly supported Euler characteristic. Even tough we have not defined this for non projective varieties, as stated in the introduction, we only need to consider projective varieties. This was already pointed out in \cite{arcila-maya_compactly_2022} and they stated the following:
		
		\begin{proposition}[cf. \cite{arcila-maya_compactly_2022} Thm. 1.13]
		\label{morphism}
			Let $k$ be a field with $char(k) = 0$. Then the compactly supported $\mathbb A^1$-Euler characteristic is well defined and the following map:
			$$
			\chi^{\mathbb A^1}_{c}:K_0(\var(k))\to GW(k)
			$$
			is a homomorphism of rings. 
		\end{proposition}

The proof of the above proposition simply follows from the usual properties of Euler characteristics, which are compatible with the relations in $K_0(\var(k))$. The morphism from Proposition \ref{morphism} can be extended to the localization of $K_0(\var(k))$ in the same way as the topological Euler characteristics (see Section \ref{extend}) after adjoining a square root of $\langle-1\rangle$ to $GW(k)$.
	\begin{equation}
		\label{morphsimA1}
		\chi_c^{\mathbb A^1}:\mathcal M_k \to GW(k)(\alpha)
	\end{equation}	
	where $\alpha$ is such that $\alpha^{2} = \chi_c^{\mathbb A^1}(\mathbb L) = \langle-1\rangle$

The morphism above allows us to get a "numerical"  version of DT invariants over any field, given that the motivic version defined in \ref{virtclassdef} is in $\mathcal M_k$ (seen as a subring of $\mathcal M^{\hat \mu}_k$). We call these invariants {\it arithmetic DT invariants.} One could also consider an equivariant version of the $\mathbb A^1$-Euler characteristic in order to define arithmetic invariants even when the virtual classes are not in $\mathcal M_k$. 

An interesting question is whether there is a direct definition of such invariants which is not related to the Grothendieck ring of varieties but defined directly using the tools of $\mathbb A^1$-homotopy theory. 

\subsection{Local \texorpdfstring{$\mathbb A^1$}{A1}-degree and Eisenbud-Khimshiashvili-Levine classes}
\label{degree} We now discuss some aspects of the Eisenbud-Khimshiashvili-Levine (EKL) classes, which, as we stated in the introduction and in section \ref{localmotivicnearby}, are an important ingredient for possible physical interpretations and might be key to finding a relationship between the motivic nearby cycles and the $\mathbb A^1$-version of the Milnor number. 

In the papers \cite{Eisenbud_Levine_1977} and \cite{khimshiashvili1977local}, the idea was to find a way of computing the local topological degree of $P:\mathbb R^n\to \mathbb R^n$ from algebraic information in the local ring of $P$ at 0. They showed that there was a quadratic form (the EKL form) defined on the local algebra whose signature corresponded to the local degree. The same problem was considered for holomorphic maps, and the rank of the EKL form ended up being equal to the local degree.

The rank and signature appearing as above is already a hint that, over an arbitrary field $k$, the class of the EKL form (EKL class) in $GW(k)$ should be equal to the refinement of the local topological degree introduced in definition \ref{deg}. Using such refinement, J.\ Kass and K.\ Wickelgren showed that the class of the EKL form in $GW(k)$ correspond to the local $\mathbb A^1$-degree \cite{wickelgrenEisenbud}. We now give the definition of the EKL form.
\begin{definition}
Consider a morphism $P:\mathbb A^n\to \mathbb A^n$ with $P(0) = 0$. Assume that this zero is isolated. We can write $P$ as $(P_1,\dots,P_n)$, with $P_i = \sum_{j=1}^n a_{ij}x_j$. Let $A=k[x_1,\dots,x_n]_{(x_1,\dots,x_m)}/(P)$ be the local algebra of $P$ at $0$. Define $E = \text{det}(a_{ij})$, where $a_{ij}$ are such that $f_i = \sum_j a_{ij}x_j$. 
       
Let $\phi:A\to k$ be any $k$-linear map and define the bilinear form $\beta_{\phi}(p,q) = \phi(pq)$. The Eisenbud-Khimshiashvili-Levine (EKL) class of $P$ is the class of $\beta_{\phi}\in GW(k)$ for any $\phi$ with $\phi(E) = 1$. It is denoted $w_0(P)$.
    \end{definition}
    
\begin{remark}
    $E$ basically carries the same information as the Jacobian determinant ${\rm det}J$ of $P$. Specifically, ${\rm det}J = \dim_k A \cdot E$. In particular, in characteristic zero, one can consider $J$ instead of $E$.
\end{remark}

    \begin{remark}\label{simple-zero}
        If $P$ has a simple zero at the origin, then $w_0(f)$ is simply the class $\langle E\rangle = \langle {\rm det}J\rangle$. This follows from the fact that in this case $A\cong k$. In particular, it corresponds to the local degree in this case, since $P$ having a simple zero implies $P$ is étale. 
     \end{remark}

    EKL classes are related to critical loci and Milnor fibres for isolated singularities. For $f:\mathbb A^n\to k$, the derivative of $f$ gives us a map $P:=df:\mathbb A^n\to \mathbb A^n$ as we considered above. Then, the EKL class $w_0(P)$ refines the Milnor number of the singularity. Indeed, over $\mathbb C$, the Milnor number is the vector space dimension of the quotient $A$, which is, by definition, the rank of the EKL quadratic form. This refinement was introduced in \cite[Section 6]{wickelgrenEisenbud}. 
    
    Over $\mathbb C$, the Milnor number of $f$ is closely related to the topology of the Milnor fibre, which is homotopic to a bouquet of $\mu$ spheres $S^n$, where $\mu$ is the Milnor number of $f$. This gives us the classical formula:
    $$
    \chi(F) = 1 + (-1)^{n-1}\mu(f),
    $$
where $F$ is the Milnor fibre of $f$. We suspect that this formula generalizes.

\begin{example}
    In the case considered in \ref{mfex}, we had $f:\mathbb A^2\to k$ given by $x^2 - y^2$. Its derivative is given by $$P:=df:\mathbb A^2\to \mathbb A^2$$ $$(x,y)\mapsto (2x,-2y).$$

    The Jacobian matrix, in this context, is simply the Hessian of $f$. This gives us:
    $$
    J(x,y) = \begin{bmatrix}
        2 & 0 \\
        0 & -2\\
    \end{bmatrix},
    $$
    for any $(x,y)$.
    
    Now, the $\mathbb A^1$-Milnor number can be easily computed by the consideration on Remark \ref{simple-zero} and Equation \ref{degree}. 
    $$
    \mu^{\mathbb A^1} = \langle {\rm det} J(0)\rangle = \langle -4\rangle =\langle-1\rangle \in GW(k).
    $$

    Finally,
    $$
    \chi^{\mathbb A^1}_c(S_{f,0}) =\chi^{\mathbb A^1}_c(1-\Lm) = \langle1\rangle - \langle-1\rangle = \langle1\rangle + (-\langle1\rangle)^{2-1}\cdot \langle-1\rangle = \langle1\rangle + (-\langle1\rangle)^{n-1}\mu^{\mathbb A^1}(f).
    $$
\end{example}

Of course, in general, one has to consider an equivariant version of the $\mathbb A^1$-Euler characteristic, which takes into account the action of $\hat \mu$, since the classes $S_{f,x}$ are in $\mathcal M_k^{\hat\mu}$. Many interesting questions can be asked regarding this topic: what is the general relationship between the $\mathbb A^1$-Milnor number of a function $f$ and the $\mathbb A^1$-Euler characteristic of the motivic class $S_f$? Is there anything that can be said for non isolated singularities? Some of these questions were studied in \cite{Azouri_2022}.


\section{Degree zero invariants of \texorpdfstring{$\mathbb A^3$}{A3}}
	
\subsection{The Hilbert Scheme of points of the affine space as a critical locus} We recall how to realize the Hilbert scheme of $n$ points  of $\mathbb A^3$ as a critical locus of a function. 
    
    Fix the notation  ${\rm Hilb}^n(\mathbb A^k)$ for the Hilbert scheme of $n$ points of $\mathbb A^k$ and
    $$
    \mathbb A^k = {\rm Spec\ }k[x_1,\dots,x_k] = {\rm Spec\ }k[x]. 
    $$

    The following considerations can be found in the notes by Nakajima \cite{Nakajima_1999}.
    
    In the affine space, a subscheme $B\subset\mathbb A^k$ of dimension 0 and degree $n$ correspond to a quotient of $k[x]$ of dimension $n$. Therefore, if we fix a vector space $V_n$ of dimension $n$, the structure we need to add in order to get a module consists of an action of $k[x]$ and an element $1\in V_n$ which generates the whole space under the action. An action of $k[x]$ is the choice of $k$ elements of ${\rm Hom}(V,V)$ which commute. This allows us consider the set:
    $$
   B(V) =  \{(A_1,\dots,A_k,v)\in {\rm Hom}^k(V,V)\times V\ |\ [A_i,A_j] = 0,\, \text{ $v$ generates $V$ under the action} \}.
    $$

    To get the Hilbert Space from $B(V)$, we need to mod by the action of $GL(V)$ (by conjugation on the $A_i$). This implies that:
    \begin{proposition}
    The Hilbert Scheme of points of $\mathbb A^k$ can be represented as
    \begin{equation}
{\rm Hilb}^n(\mathbb A^k) \cong\left\{\left(A_1,\dots,A_k, v\right) \middle| \begin{aligned}
&\text { (i) }\left[A_i, A_j\right]=0 \text{ for all $i, j$} \\
&\text { (ii) } \text{$v$ generates $V$ under the action of the $A_j$}\\
\end{aligned}\right\} \Bigg/ \mathrm{GL}_n(k),
\end{equation}
\end{proposition}

In the case $k=3$, this space can be seen as a critical locus. This is done by considering the map 
$f_n: (A,B,C,v)\mapsto {\rm Tr}([A,B]C)$ on the smooth space $$
M_n=\left\{\left(A_1,\dots,A_k, v\right) \mid 
 \text{$v$ generates $V$ under the action of the $A_j$}\\
\right\} / \mathrm{GL}_n(k),
$$
    
\begin{proposition} ${\rm Hilb}^n(\mathbb A^3) = \{df_n = 0\}\subset M$
        
    \end{proposition}
    \begin{proof}
                
        By our discussion above, it is enough to show that the condition $df_n = 0$ correspond to commutativity of the $A_i$. Indeed, we have:
        $$
        {\rm Tr}([A,B]C) = \sum_{i}\sum_k\sum_j(a_{ij}b_{jk} - b_{ij}a_{jk})c_{ki}.
        $$
        
        Therefore, if the derivatives with respect to each entry of $C$ are all zero, we get that $[A,B] = 0$. As ${\rm Tr}([A,B]C) = {\rm Tr}(A[B,C]) = {\rm Tr}(B[C,A])$, the vanishing of the other derivatives implies that the other pairs of matrices commute.
    \end{proof}

\subsection{The virtual classes of the Hilbert Scheme of Points } To compute the virtual classes over $\mathbb C$, the authors of \cite{behrend_motivic_2013} used Proposition \ref{torus} and the fact that there is a natural toric action on $M$ which descends given by
$$
(t_1,t_2,t_3)\cdot (A,B,C,v) \mapsto (t_1A,t_2B,t_3C,t_1t_2t_3v)
$$
which satisfies all the hypothesis. For details, see \cite[Lemma 2.4]{behrend_motivic_2013}.

This implies that the virtual class of ${\rm Hilb}^n(\mathbb C^3)$ can be computed by the difference $[f_n^{-1}(1)]-[f_n^{-1}(0)]$. This difference was computed in \cite{behrend_motivic_2013} to correspond to a generating series
		\begin{equation} 
			\label{partition}
			Z_{\mathbb C^3}(t) = \sum_{n=0}^{\infty}[{ \rm Hilb}(\mathbb C^3)]_{{\it vir}}t^n = \prod_{m=1}^{\infty}\prod_{k=0}^{m-1}(1-\mathbb L^{k+2-m/2}t^m)^{-1}.
		\end{equation}

  \begin{remark}
      The computation itself mainly relies on the motivic classes of Grassmanians, general linear groups and of the variety of commuting matrices, which are the same over any field. The only difficulty in generalizing this computation would be to prove that the virtual class is given by the difference of the fibers. Sample targets for generalization include examples studied by Choi-Katz-Klemm in \cite{MR3201216}.
  \end{remark}

Applying the morphism in equation \ref{morphsimA1}, we can get arithmetic DT invariants as defined in Section \ref{A1}.
	\begin{proposition}
	     The DT invariants for $\mathbb A^3$  can be refined over $GW(k)$ by the generating series:
 	\begin{equation}
 			\prod_{n=1}^{\infty} (\langle 1\rangle -(\alpha t)^{2n-1})^{-1}\prod_{n=1}^{\infty}(\langle1\rangle - (\alpha t)^nH + \langle-1\rangle(\alpha t)^{2n})^{-\lfloor\frac n2\rfloor}
 	\end{equation}

	\end{proposition}
\begin{proof}
Applying the morphism, we first just send $\mathbb L$ to $\langle-1\rangle$ and 
$\mathbb L^{\frac12}$ to $\alpha$.
$$
\prod_{m=1}^{\infty}\prod_{k=0}^{m-1}(\langle 1\rangle- \langle-1\rangle^{k+2}\alpha^{-m}t^m)^{-1}	
$$
Notice that
$$
(\langle 1\rangle-\langle-1\rangle\alpha^{m}t^m)(\langle 1\rangle-\alpha^{m}t^m) = 
(\langle1\rangle - \alpha^mt^mH + \langle-1\rangle\alpha^{2m}t^{2m})
$$
where $H$ is the hyperbolic form given by $\langle1\rangle+\langle-1\rangle$. 
	
In the case $m$ is even, the product above will appear exactly $\frac m2$ times. For $m$ odd, it 
appears $\frac{m-1}2 =\lfloor \frac m2\rfloor $ times and we get an extra factor of the form 
$(\langle 1\rangle -\alpha^mt^m)^{-1}$. 
$$
\prod_{m=1}^{\infty}\prod_{k=0}^{m-1}(\langle 1\rangle- \langle-1\rangle^{k+2}\alpha^{m}t^m)^{-1} 
= \prod_{m \text{ odd}} (\langle 1\rangle -\alpha^mt^m)^{-1}	\prod_{m=1}^{\infty}(\langle1
\rangle - (\alpha t)^mH + \langle-1\rangle(\alpha t)^{2m})^{-\lfloor\frac m2\rfloor}
$$
By making $m = 2n - 1$ in the first product and $m = n$ in the second, we get:
\begin{equation}
\prod_{n=1}^{\infty} (\langle 1\rangle -(\alpha t)^{2n-1})^{-1}\prod_{n=1}^{\infty}
(\langle1\rangle (\alpha t)^nH + \langle-1\rangle(\alpha t)^{2n})^{-\lfloor\frac n2\rfloor}
\end{equation}
\end{proof}

\begin{remark}
The above refinement is compatible with previous results over $\mathbb R$ and $\mathbb C$.
\end{remark}

Taking $k=\mathbb C$, we have $GW(k) = \mathbb Z$ and $\alpha = -1$, which results in the classical 
MacMahon generating function for the number of plane partitions:
	
$$
\prod_{n=1}^{\infty} (1 -(-t)^{2n-1})^{-1}\prod_{n=1}^{\infty}(1 - 2(-t)^n + 
(-t)^{2n})^{-\lfloor\frac n2\rfloor} =
$$
$$
\prod_{n=1}^{\infty} (1 -(-t)^{2n-1})^{-1}\prod_{n=1}^{\infty}(1 - 
(-t)^n)^{-2\lfloor\frac n2\rfloor} = \prod_{n=1}^{\infty}(1 - (-t)^n)^{-n} = M(-t)
$$

For $k=\mathbb R$, after computing the signature morphism (i.e., sending $\langle 1\rangle$ to $1$, 
$\langle-1\rangle$ to $-1$ and taking $\alpha = i$), we get the symmetric MacMahon function, 
counting symmetric plane partitions, which correspond to the real count as computed by Pasquetti-Krefl-Walcher 
in \cite{krefl_real_2010}.
$$
\prod_{n=1}^{\infty} (1 -(-i t)^{2n-1})^{-1}\prod_{n=1}^{\infty}(1 - 0(-i t)^n + 
(-i t)^{2n})^{-\lfloor\frac n2\rfloor} =
$$
$$
\prod_{n=1}^{\infty} (1 -(-i t)^{2n-1})^{-1}\prod_{n=1}^{\infty}(1 - 
(-i t)^{2n})^{-\lfloor\frac n2\rfloor} = M^{\text{sym}}(-i t) 
$$


\subsection{Real and refined Gopakumar-Vafa invariants at the Castelnuovo bound}

Our interest now turns to the computation of Gopakumar-Vafa invariants for $M$ a smooth quintic hypersurface 
in $\mathbb P^4$, through their relation to DT invariants. In this case, the moduli spaces of interest are $I_n(M,d)$, 
which correspond 
to the Hilbert scheme parameterizing subschemes of $M$ with Hilbert polynomial given by $dt+n$, that is, curves 
of degree $d$ and arithmetic genus $1-n$. The GV invariants $n^d_g$ correspond to DT invariants $I_{1-g,d}$. Recent work 
by Liu-Ruan \cite{liu_castelnuovo_2022} and Alexandrov-Feyzbakhsh-Klemm-Pioline-Schimannek \cite{alexandrov2023quantum}
has established the famous Castelnuovo bound for Gopakumar-Vafa invariants, which was predicted in physics,
$$
n^d_g = I_{1-g,d} = 0,\text{ for any  $d$ and $g$ with }  g>\frac{d^2 + 5d+10}{10}=:B(d).
$$
and led to a computation of the numbers $n^d_{B(d)} = I_{1-B(d),d}$. That is, the numbers $n^d_g$ when the 
pair $(g,d)$ is on the bound. Here, we write formulae for the motivic and arithmetic refinements of such 
numbers at the bound. This was done using the fact that, for $n=B(d)$, $\mathcal M_{n,d}$ is not only 
smooth but a projective bundle over a projective space (see \cite[Prop. 6.2]{liu_castelnuovo_2022}, 
where they prove it over $\mathbb C$).We believe that this is true over any $k$ of characteristic zero.

\begin{proposition}
    Let $g = B(d)$. This implies that $B(d)$ is an integer, which means that $d$ can be written as $d=5m$ for some $m$. Assuming that we can write the moduli space as a projective bundle as above, motivic and arithmetic refinements of the invariants $I_{n,d}$ for $n = 1-g =1 - B(d)$, are given by the formulae:
$$
[\mathcal M_{n,d}]_{{\it vir}} =\Lm^{\frac N2+2}\frac{(\Lm^{N+1} - 1)(\Lm^5 - 1)}{(\Lm-1)^2}\in\mathcal M_k
$$
and, applying the morphism $\chi^{\mathbb A^1}$:
$$
\chi^{\mathbb A^1}\left([\mathcal M_{n,d}]_{{\it vir}}\right) =   \left\{\begin{aligned}&\alpha\cdot\left(\frac{6+5N}2\langle1\rangle + \frac{4+5N}2\langle-1\rangle\right), \text{ for } m=0,1\, {\rm mod}\, 4 \\ &\frac{5(N+1)}2H,\text{ for } m=2,3\, {\rm mod}\, 4\end{aligned}\right.
$$
 where $N=\binom{m+3}{3} -\binom{m-2}{3}-1$.
\end{proposition}

\begin{proof}
    The base is $\mathbb P^4$ and the fibre is $\mathbb P^N$, where $N=\binom{m+3}{3} -\binom{m-2}{3}-1$.
    
    Then, the first part of the result is simply given by using the fact that, for smooth varieties, the motivic virtual class is simply the class of the variety times $\Lm^{-\dim \mathcal M_{n,d}/2}$ (see Definition \ref{virtclassdef} and Example \ref{smoothvirtual}). To compute $[\mathcal M_{n,d}]$, we use that the class of a fibre bundle is the product of the fibre and the base.

    Using that
    $$
    [\mathbb P^r] = \Lm^r + \Lm^{r-1} + \dots \Lm+ 1 = \frac{\Lm^{r+1}}{\Lm-1}
    $$
    and that $\dim\mathcal M_{d,n} = N+4$, we get
    $$
     [\mathcal M_{n,d}]_{vir} =\Lm^{-\frac{\dim\mathcal M_{n,d}}2}[\mathcal M_{n,d}] = \Lm^{\frac N2+2}[\mathbb P^N]\cdot [\mathbb P^4] =\Lm^{\frac N2+2}\frac{(\Lm^{N+1} - 1)(\Lm^5 - 1)}{(\Lm-1)^2}\in\mathcal M_k.
     $$

     Finally, applying the morphism, we only need to keep track of whether $N$ is even or odd. If it is even, the $\mathbb A^1$-Euler characteristic of $\mathbb P^N$ is $\frac{N+1}2H$ and if it is odd it is given by $\frac{N+1}2\langle1\rangle + \frac{N-1}{2}\langle-1\rangle$. 
\end{proof}

The formula above, specially the one in $GW(k)$, gives a prediction of what should be the $\mathbb A^1$-count of curves of higher genus on the quintic over any field. It would be interesting to check if this result can be reached with direct methods, without making use of motivic DT invariants.


\bibliographystyle{abbrv}
\bibliography{MotivicAndArith.bib}

\begin{thebibliography}{10}

\bibitem{alexandrov2023quantum}
S.~Alexandrov, S.~Feyzbakhsh, A.~Klemm, B.~Pioline, and T.~Schimannek.
\newblock Quantum geometry, stability and modularity, 2023.

\bibitem{arcila-maya_compactly_2022}
N.~Arcila-Maya, C.~Bethea, M.~Opie, K.~Wickelgren, and I.~Zakharevich.
\newblock Compactly supported {$\mathbb A^1$}-{Euler} characteristic and the
  {Hochschild} complex.
\newblock {\em Topology and its Applications}, 316:108108, July 2022.

\bibitem{Azouri_2022}
R.~Azouri.
\newblock Motivic euler characteristic of nearby cycles and a generalized
  quadratic conductor formula, Aug 2022.
\newblock arXiv:2101.02686 [math].

\bibitem{barge_groupe_2000}
J.~Barge and F.~Morel.
\newblock Groupe de {Chow} des cycles orientés et classe d'{Euler} des fibrés
  vectoriels.
\newblock {\em C. R. Acad. Sci. Paris Sér. I Math.}, 330(4):287--290, 2000.

\bibitem{Behrend_2009}
K.~Behrend.
\newblock Donaldson-thomas type invariants via microlocal geometry.
\newblock {\em Annals of Mathematics}, 170(3):1307–1338, Nov 2009.

\bibitem{behrend_motivic_2013}
K.~Behrend, J.~Bryan, and B.~Szendrői.
\newblock Motivic degree zero {Donaldson}–{Thomas} invariants.
\newblock {\em Invent. math.}, 192(1):111--160, Apr. 2013.

\bibitem{bittner_universal_2004}
F.~Bittner.
\newblock The universal {Euler} characteristic for varieties of characteristic
  zero.
\newblock {\em Compositio Mathematica}, 140(4):1011--1032, July 2004.
\newblock Publisher: London Mathematical Society.

\bibitem{MR3201216}
J.~Choi, S.~Katz, and A.~Klemm.
\newblock The refined {BPS} index from stable pair invariants.
\newblock {\em Comm. Math. Phys.}, 328(3):903--954, 2014.

\bibitem{SGA7}
P.~Deligne and N.~Katz.
\newblock {\em Groupes de Monodromie en Géométrie Algébrique dirigé par A.
  Grothendieck}, volume 340 of {\em Lecture Notes in Mathematics}.
\newblock Springer, Berlin, Heidelberg, 1973.

\bibitem{Denef_Loeser_2002}
J.~Denef and F.~Loeser.
\newblock Lefschetz numbers of iterates of the monodromy and truncated arcs.
\newblock {\em Topology}, 41(5):1031–1040, Sep 2002.

\bibitem{Eisenbud_Levine_1977}
D.~Eisenbud and H.~I. Levine.
\newblock An algebraic formula for the degree of a {$C^\infty$} map germ.
\newblock {\em Annals of Mathematics}, 106(1):19–44, 1977.

\bibitem{Goettsche_Shende_2014}
L.~Göttsche and V.~Shende.
\newblock Refined curve counting on complex surfaces.
\newblock {\em Geometry \& Topology}, 18(4):2245–2307, Oct 2014.

\bibitem{Kass_Levine_Solomon_Wickelgren_2023}
J.~L. Kass, M.~Levine, J.~P. Solomon, and K.~Wickelgren.
\newblock A quadratically enriched count of rational curves, Jul 2023.
\newblock arXiv:2307.01936 [math].

\bibitem{wickelgrenEisenbud}
J.~L. Kass and K.~Wickelgren.
\newblock The class of {E}isenbud-{K}himshiashvili-{L}evine is the local
  {$\mathbf{A}^1$}-{B}rouwer degree.
\newblock {\em Duke Math. J.}, 168(3):429--469, 2019.

\bibitem{kass_arithmetic_2021}
J.~L. Kass and K.~Wickelgren.
\newblock An arithmetic count of the lines on a smooth cubic surface.
\newblock {\em Compositio Mathematica}, 157(4):677--709, Apr. 2021.
\newblock Publisher: London Mathematical Society.

\bibitem{khimshiashvili1977local}
G.~Khimshiashvili.
\newblock On the local degree of a smooth mapping.
\newblock {\em Bull. Acad. Sci. Georgian SSR}, 85:309--312, 1977.

\bibitem{kontsevich_stability_2008}
M.~Kontsevich and Y.~Soibelman.
\newblock Stability structures, motivic {Donaldson}-{Thomas} invariants and
  cluster transformations, Nov. 2008.
\newblock arXiv:0811.2435 [hep-th].

\bibitem{krefl_real_2010}
D.~Krefl, S.~Pasquetti, and J.~Walcher.
\newblock The real topological vertex at work.
\newblock {\em Nuclear Physics B}, 833(3):153--198, July 2010.

\bibitem{krefl2010}
D.~Krefl and J.~Walcher.
\newblock Extended holomorphic anomaly in gauge theory.
\newblock {\em Letters in Mathematical Physics}, 95(1):67--88, oct 2010.

\bibitem{levine_aspects_2020}
M.~Levine.
\newblock Aspects of enumerative geometry with quadratic forms.
\newblock {\em Documenta Mathematica}, Vol 25:2179--2239 Pages, Dec. 2020.
\newblock arXiv:1703.03049 [math].

\bibitem{levine_atiyah-bott_2022}
M.~Levine.
\newblock Atiyah-{Bott} localization in equivariant {Witt} cohomology, Apr.
  2022.
\newblock arXiv:2203.13882 [math].

\bibitem{levine_quadratic_2022}
M.~Levine and S.~Pauli.
\newblock Quadratic {Counts} of {Twisted} {Cubics}, June 2022.
\newblock arXiv:2206.05729 [math].

\bibitem{Levine_Raksit_2020}
M.~Levine and A.~Raksit.
\newblock Motivic gauss–bonnet formulas.
\newblock {\em Algebra \& Number Theory}, 14(7):1801–1851, Aug 2020.

\bibitem{liu_castelnuovo_2022}
Z.~Liu and Y.~Ruan.
\newblock Castelnuovo bound and higher genus {Gromov}-{Witten} invariants of
  quintic 3-folds, Oct. 2022.
\newblock arXiv:2210.13411 [hep-th].

\bibitem{Maulik_2006}
D.~Maulik, N.~Nekrasov, A.~Okounkov, and R.~Pandharipande.
\newblock Gromov–witten theory and donaldson–thomas theory, i.
\newblock {\em Compositio Mathematica}, 142(5):1263–1285, Sep 2006.

\bibitem{morel_1-homotopy_1999}
F.~Morel and V.~Voevodsky.
\newblock A{\textasciicircum}1-homotopy theory of schemes.
\newblock {\em Inst. Hautes Études Sci. Publ. Math.}, pages 45--143 (2001),
  1999.

\bibitem{Nakajima_1999}
H.~Nakajima.
\newblock {\em Lectures on Hilbert schemes of points on surfaces}.
\newblock University lecture series. American Mathematical Society, Providence,
  R.I, 1999.

\bibitem{thirteenhalves}
R.~Pandharipande and R.~P. Thomas.
\newblock 13/2 ways of counting curves.
\newblock In {\em Moduli spaces}, volume 411 of {\em London Math. Soc. Lecture
  Note Ser.}, pages 282--333. Cambridge Univ. Press, Cambridge, 2014.

\bibitem{pauliquintic}
S.~Pauli.
\newblock Quadratic types and the dynamic {E}uler number of lines on a quintic
  threefold.
\newblock {\em Adv. Math.}, 405:Paper No. 108508, 37, 2022.

\bibitem{Roendigs_2016}
O.~Röndigs.
\newblock The {G}rothendieck ring of varieties and algebraic {K}-theory of
  spaces, Nov 2016.
\newblock arXiv:1611.09327 [math].

\bibitem{solomonintersection}
J.~P. Solomon.
\newblock Intersection theory on the moduli space of holomorphic curves with
  lagrangian boundary conditions, 2006.

\bibitem{walcher2007}
J.~Walcher.
\newblock Opening mirror symmetry on the quintic.
\newblock {\em Communications in Mathematical Physics}, 276(3):671--689, oct
  2007.

\bibitem{walcher2012arithmetic}
J.~Walcher.
\newblock On the arithmetic of {D}-brane superpotentials. {L}ines and conics on
  the mirror quintic.
\newblock {\em Commun. Number Theory Phys.}, 6(2):279--337, 2012.

\end{thebibliography}

	\end{document}